\documentclass{amsart}

\title{Edge ideals and DG algebra resolutions}

\subjclass[2010]{Primary: 16E45; Secondary: 16S37, 13D02, 05C38.}

\keywords{  DG algebra resolution, Koszul homology, acyclic closure, minimal model, deviations, Poincar\'e series, Hilbert series, Koszul algebra, edge ideal, paths and cycles.}

\author[Adam Boocher]{Adam Boocher}
\address{Adam Boocher \\ School of Mathematics \\ University of Edinburgh \\James Clerk Maxwell Building, Mayfield Road \\Edinburgh EH9 3JZ, Scotland}
\email{adam.boocher@ed.ac.uk}
 
 \author[Alessio D'Al\`{i}]{Alessio D'Al\`{i}}
\address{Alessio D'Al\`{i} \\ Dipartimento di Matematica \\ Universit\`{a} degli Studi di Genova \\Via Dodecaneso 35 \\16146 Genova, Italy}
\email{dali@dima.unige.it}
 
 \author[Elo\'isa Grifo]{Elo\'isa Grifo}
\address{Elo\'isa Grifo \\Department of Mathematics \\  University of Virginia \\141 Cabell Drive, Kerchof Hall \\Charlottesville, VA 22904, USA}
\email{er2eq@virginia.edu}

\author[Jonathan Monta\~{n}o]{Jonathan Monta\~{n}o}
\address{Jonathan Monta\~{n}o \\ Department of Mathematics \\ Purdue University \\150 North University Street \\West Lafayette, IN 47907, USA}
\email{jmontano@purdue.edu}

\author[Alessio Sammartano]{Alessio Sammartano}
\address{Alessio Sammartano \\ Department of Mathematics \\ Purdue University \\150 North University Street \\West Lafayette, IN 47907, USA}
\email{asammart@purdue.edu}

\usepackage{amsmath,color}
\usepackage{amssymb}
\usepackage{amsfonts}
\usepackage{graphicx}
\usepackage{mathrsfs}
\usepackage{amscd}
\usepackage[all]{xy}
\usepackage{hyperref}
\usepackage{amsthm}
\usepackage{fancyhdr}
\usepackage{verbatim}
\usepackage{amscd}
\usepackage{mathtools}
\usepackage{tikz}

\newtheorem{thm}{\bf Theorem}[section]

\newtheorem{prop}[thm]{\bf Proposition}
\newtheorem{lemma}[thm]{\bf Lemma}

\theoremstyle{definition}
\newtheorem{definition}[thm]{\bf Definition}
\theoremstyle{remark}
\newtheorem{remark}[thm]{\bf Remark}

\newtheorem{question}[thm]{\bf Question}
\newtheorem{example}[thm]{\bf Example}
\numberwithin{equation}{section}

\newcommand{\be}[2]{\beta_{#1,#2}}

\def\Supp{\operatorname{Supp\, }}

\def\Tor{\operatorname{Tor}}

\def \Im{{\operatorname{Im\,}}}

\def \Ho{{\operatorname{H \,}}}
\def \Card{{\operatorname{Card}}}
\def \id{{\operatorname{id}}}

\DeclareMathOperator{\HS}{HS}
\DeclareMathOperator{\mdeg}{mdeg}

\newcommand{\ve}{\varepsilon}

\def \fv{\mathbf v}
\def \fw{\mathbf w}
\def \fu{\mathbf u}
\def \f1{\mathbf{1}}

\def\xi{x}
\def\fxi{\mathbf{\xi}}
\def\zi{z}

\def\wx{\widetilde{x}}
\def\ls{\leqslant}
\def\gs{\geqslant}

\def\ee{\varepsilon}
\def\fn{\mathfrak{n}}

\def\fa{\mathfrak{a}}

\def \RR{\mathbb R}

\def \NN{\mathbb N}
\def \ZZ{\mathbb Z}
\def \FF{\mathbb F}

\def \C{\mathcal C}

\def \P{\mathcal P}

\newcommand\pat[1]{I(\mathcal{P}_{#1})}
\newcommand\cyc[1]{I(\mathcal{C}_{#1})}

\begin{document}

\maketitle

% Abstract
\begin{abstract}
Let $R=S/I$  where $S=k[T_1, \ldots, T_n]$ and $I$ is a homogeneous ideal in $S$.
The acyclic closure $R\langle Y \rangle$ of $k$ over $R$ is a  DG algebra resolution obtained by means of Tate's process of adjoining variables to kill cycles.
In a similar way one can obtain  the minimal model $S[X]$, a DG algebra resolution of $R$ over $S$.
By a theorem of Avramov  there is a tight connection between these two resolutions.
In this paper we study these two resolutions when $I$ is the edge ideal of a path or a cycle.
We determine the behavior of the deviations  $\varepsilon_i(R)$, which are the number of variables in $R\langle Y \rangle$ in  homological degree $i$.
We apply our results to the study of the $k$-algebra structure of the Koszul homology of $R$.
\end{abstract}

% The article itself

\section{Introduction}

Let $S= k[T_1, \ldots, T_n]$, $I\subseteq (T_1,\ldots, T_n)^2$ be a homogeneous ideal and  $R  = S/I$.
Endowing free resolutions over $R$ with  multiplicative structures can be  a powerful technique  in studying homological properties of the ring.
The idea of multiplicative free resolution  is made precise by the notion of a {\bf Differential Graded (DG) algebra resolution} (cf. \cite[Ch. 31]{PeevaGradedSyzygies}). 
Several interesting resolutions admit a DG algebra structure: examples include the Koszul complex, the Taylor resolution of monomial ideals, 
the Eliahou-Kervaire resolution (cf. \cite{0Borel}),
the minimal free resolution of $k$ (cf. \cite{Gulliksen}, \cite{Schoeller}),
and
free resolutions of length at most 3 (cf. \cite{BuchEis}).
In general, though, minimality and DG algebra structure are incompatible conditions on  resolutions of an $R$-algebra:
obstructions were discovered and used in \cite{AvramovObstructions} to produce perfect ideals $\fa\subseteq R$ with prescribed grade $\gs 4$ 
such that the minimal free $R$-resolution of  $R/\fa$ admits no DG algebra structure.
 
Nevertheless, it is  always possible to obtain  DG algebra resolutions of a factor ring $R/\fa$
by a recursive process that mimics the construction of the minimal free resolution of a module;
 we refer to \cite{Avramov6Lectures} for more details and background.
Let $\{a_1, \ldots, a_r\}$ be a minimal generating set of $\fa$ and start with the Koszul complex on $a_1, \ldots, a_r$. 
Apply inductively Tate's process  of adjoining variables  in homological degree $i+1$ to kill cycles in homological degree $i$ whose classes generate the $i$-th homology minimally (cf. \cite{Tate}).
Using exterior variables to kill cycles of even degrees and polynomial variables to kill cycles of odd degrees we obtain a DG algebra resolution of $R$, called a {\bf minimal model}  of $R/\fa$ over $R$
and denoted by $R[X]$, where $X$ is the collection of all the variables adjoined during the process (cf. \cite[7.2]{Avramov6Lectures}).
Using divided power variables instead of polynomial variables we obtain another DG algebra resolution of $R$, called an {\bf acyclic closure} of $R/\fa$ over $R$ and denoted by $R \langle Y \rangle$;
similarly, $Y$ is the collection of all the {variables} adjoined (cf. \cite[6.3]{Avramov6Lectures}).
Both objects are uniquely determined up to isomorphisms of DG algebras. 
The minimal model and the acyclic closure are isomorphic if $R$ is a complete intersection or if $\mathbb{Q}\subseteq R$, but they  differ in general.
The set  $X_i$ (resp. $Y_i$)  of variables adjoined to $R[X]$ (resp. $R\langle Y \rangle$) in homological degree $i$ has finite cardinality.

A  result of Avramov relates the minimal model $S[X]$ of  $R$ over the polynomial ring $S$ to the acyclic closure $R\langle Y \rangle$ of the residue field $k$ over $R$:
the equality $\Card(X_i)=\Card(Y_{i+1})$ holds for all $i\gs 1$ (cf. \cite[7.2.6]{Avramov6Lectures}).
We remark that such resolutions are considerably hard to describe explicitly. 
The growth of $S[X]$ and $R\langle Y \rangle$ is determined by the integers $\varepsilon_i(R) = \Card (Y_{i})$, known as the {\bf deviations} of $R$
 (because they measure how much $R$ deviates from being regular or a complete intersection, cf. \cite{AvramovCI}, \cite[Section 7.3]{Avramov6Lectures}).
The deviations are related to the  Poincar\'e series $P^R_k(z) = \sum_{i\gs 0} \dim_k \Tor^R_i(k,k)z^i$
by  the following formula  (cf. \cite[7.1.1]{Avramov6Lectures})
\begin{equation}
P^R_k(z) = {\frac{     \prod_{i \in 2\mathbb{N}+1}    (1+z^i)^{\varepsilon_i(R)}  }{ \prod_{i \in 2\mathbb{N}}    (1-z^i)^{\varepsilon_i(R)}     }    }.
\end{equation}

In this paper we study the minimal model $S[X]$ of $R$ over $S$ and the acyclic closure $R\langle Y \rangle$ of $k$ over $R$ when $R$ is a {\bf Koszul algebra},
i.e. when $k$ has a linear resolution over $R$.
It is well known that for Koszul algebras $R$ the Poincar\'e series is related to the Hilbert series by the equation
\begin{equation}\label{EquationHilbertPoincare}
P^R_k(z) {\HS_R(-z)} = 1.
\end{equation}
Furthermore, $R$ is Koszul if $I$ is a quadratic monomial, in particular if $I$ is the edge ideal of a graph. 
See  \cite[Ch. 34]{PeevaGradedSyzygies} and the references therein for details. 

In Section \ref{SectionEdgeIdeals} we study the deviations of $R$  when $I$ is the edge ideal of a cycle or a path.
In order to do so, we exploit the multigraded structure of $R\langle Y \rangle$.
In Theorem~\ref{TheoremExistenceSequences} we determine the  deviations $\varepsilon_i(R)$ for $i=1, \ldots, n$; these values are determined by two  sequences  $\{\alpha_s\}$ and $\{\gamma_s\}$, that are independent of the number of vertices $n$.

In Section \ref{SectionKoszulHomology} we use the minimal model $S[X]$ to investigate the  Koszul homology $H^R= \Tor^S(R,k)$ of $R$.
Its $k$-algebra structure encodes interesting homological information on $R$: 
for instance, $R$ is a complete intersection if and only if $H^R$ is an exterior algebra on $H^R_1$ (cf. \cite{Tate}),
 and $R$ is Gorenstein if and only if $H^R$ is a Poincar\'e algebra (cf. \cite{AvramovGolod}). 
When $R$ is a Golod ring $H^R$ has trivial multiplication (cf.  \cite{Golod}). 
It is not clear how the Koszul property of $R$ is reflected in the $k$-algebra structure of $H^R$. 
Results in this direction have been obtained by Avramov, Conca, and Iyengar in \cite{AvramovConcaIyengar} and \cite{AvramovConcaIyengar2}.
We extend  their theorem  \cite[5.1]{AvramovConcaIyengar2}
to show that if $R$ is Koszul then the  components of $H^R$ of bidegrees $(i, 2i-1)$ are generated in bidegrees $(1,2)$ and $(2,3)$, see Theorem \ref{KoszulSecondDiag}.
While these theorems tell us that a part of the $k$-algebra $H^R$ is generated in the linear strand if $R$ is Koszul, in general there may be minimal algebra generators in other positions (see Remark \ref{RemarkKoszulHomologyLinearStrand}). 
In fact, in Theorem \ref{KoszulHomologyPolygons} we give a complete description of the $k$-algebra generators of the Koszul homology of  the algebras considered in Section \ref{SectionEdgeIdeals}:
for edge ideals of cycles, 
the property of being generated in the linear strand depends on the residue of the number of vertices modulo 3.

\section{Deviations of edge ideals of paths and cycles}\label{SectionEdgeIdeals}

Throughout this section
we consider $S=[T_1, \ldots, T_n]$ as an $\NN^n$-graded algebra by assigning to each monomial of $S$ the multidegree
$\mdeg(T_1^{v_1}\cdots T_n^{v_n})=\fv =(v_1, \ldots, v_n)$.
If $I \subseteq S$ is a monomial ideal, 
then $R=S/I$ inherits the multigrading from $S$.  Let $t_i$ be the image of $T_i$ in $R$ and denote by $\be{i}{\mathbf{v}}^R(k) = \dim_k \Tor_i^R(k,k)_{\mathbf{v}}$ the multigraded Betti numbers of $k$ over $R$, 
and by   
$P_k^R(\zi,\fxi)=\sum_{i,\fv}\be{i}{\mathbf{v}}^R(k)\zi^i\fxi^\fv$
the multigraded Poincar\'e series of $R$, where $\fxi^\fv=x_1^{v_1}\cdots x_n^{v_n}$. 
There are uniquely determined nonnegative integers $\ee_{i,\fv}=\ee_{i,\fv}(R)$ satisfying the infinite product expansion (cf. \cite[Remark 1]{Berglund})
\begin{equation}\label{MultigradedExpansion}
P_k^R(\zi,\fxi)=\prod_{i\gs 1, \fv \in \NN^n}\frac{(1+\zi^{2i-1}\fxi^{\fv})^{\ee_{2i-1,\fv}}}{(1-\zi^{2i}\fxi^{\fv})^{\ee_{2i,\fv}}}.
\end{equation}
The numbers $\ee_{i,\fv}$ are known as the {\bf multigraded deviations} of $R$.
They refine the usual deviations in the sense that $\ee_{i} =\sum_{\fv \in \NN^n}\ee_{i,\fv}$. 
We can repeat the constructions in the Introduction
 respecting the multigrading.
In particular, 
we can construct an acyclic closure $R\langle Y \rangle$ of $k$ over $R$ and hence 
$$\ee_{i,\fv}(R)=\Card(Y_{i,\fv}),$$ 
where $Y_{i,\fv}$ denotes the set of variables in homological degree $i$ and internal multidegree $\fv$. 
Similarly, we can construct a minimal model $S[X]$ of $R$ over $S$
and denote by $X_{i,\fv}$ the variables in homological degree $i$ and internal multidegree $\fv$;
 the multigraded version of \cite[7.2.6]{Avramov6Lectures}  holds, see \cite[Lemma 5]{Berglund}. 

Let $\HS_R(\fxi)=\sum_{\fv\in\NN^n}\dim_k(R_{\fv})\fxi^{\fv}$ be the multigraded Hilbert series of $R$. 
The following fact is folklore. We include here its proof for the reader's convenience.

\begin{prop}\label{MultPoncareHilbert}
Let $S=k[T_1,\ldots, T_n]$ and $I$ be a monomial ideal of $S$. Then
$$
P^R_k(-1,\fxi)\HS_R(\fxi)=1.
$$
\end{prop}
\begin{proof}
Let $\FF$ be the augmented minimal free resolution of $k$ over $R$ and fix $\fv\in\NN^n$. Let $\FF_{\fv}$ be the strand of $\FF$ in multidegree $\fv$: $$\FF_{\fv}:\cdots\rightarrow F_{2,\fv}\rightarrow F_{1,\fv}\rightarrow F_{0,\fv}=R_\fv \rightarrow k_{\fv}\rightarrow 0.$$
Since $\FF_{\fv}$ is an exact complex of $k$-vector spaces, we have $\sum_{i\gs 0}(-1)^i\dim_k F_{i,\fv}=1$ if $\fv=(0,\ldots,0)$ 
and $0$ otherwise. 
On the other hand, it is easy to see that this alternating sum is equal to the coefficient of $\fxi^{\fv}$ in $P^R_k(-1,\fxi)\HS_R(\fxi)$ and the conclusion follows.
\end{proof}

Let $\mathcal{G}$ be a graph with vertices $\{1, \ldots, n\}$. The {\bf edge ideal} of $\mathcal{G}$ is the ideal $I(\mathcal{G})\subseteq S$
generated by the monomials $T_iT_j$ such that $\{i,j\}$ is an edge of $\mathcal{G}$.
We denote the {\bf $n$-path} by $\P_n$ and  the {\bf $n$-cycle} by $\C_n$,  
the graphs
whose edges are respectively 
$
\big\{\{1,2\},\,\ldots,\, \{n-1,n\}\big\}
$
and
$
\big\{\{1,2\},\,\ldots,\, \{n-1,n\}, \{n,1\}\big\}$ (see Figure \ref{pictures}).

\begin{figure}[ht]

\begin{tikzpicture}
\draw [fill] (-1,0) circle [radius=0.04];
\draw [fill] (-2,0) circle [radius=0.04];
\draw [fill] (-3,0) circle [radius=0.04];
\draw [fill] (-4,0) circle [radius=0.04];
\draw [fill] (-5,0) circle [radius=0.04];

\draw [fill] (1,0) circle [radius=0.04];
\draw [fill] (1.7,1) circle [radius=0.04];
\draw [fill] (1.7,-1) circle [radius=0.04];
\draw [fill] (3.1,-1) circle [radius=0.04];
\draw [fill] (3.1,1) circle [radius=0.04];
\draw [fill] (3.8,0) circle [radius=0.04];

\draw (-1,0)--(-5,0);
\draw  (1,0)--(1.7,1)--(3.1,1)--(3.8,0)--(3.1,-1)--(1.7,-1)--(1,0);

\node at (-1,-0.3) {5};
\node at (-2,-0.3) {4};
\node at (-3,-0.3) {3};
\node at (-4,-0.3) {2};
\node at (-5,-0.3) {1};

\node at (0.8,0) {1};
\node at (1.7,1.3) {2};
\node at (1.7,-1.3) {6};
\node at (3.1,-1.3) {5};
\node at (3.1,1.3) {3};
\node at (4,0) {4};

\end{tikzpicture}

\caption{The graphs $\P_5$ and $\C_6$.}
\label{pictures}
\end{figure}

Given a vector $\fv = (v_1, \ldots, v_n)\in \mathbb{N}^n$,
denote by $\|\fv\|=\sum_i v_i$ the {\bf 1-norm} of $\fv$.
If $R$ is a Koszul algebra, $\beta_{i,\fv}^R(k) \ne 0$ only if $i =\|\fv\|$, and thus a deviation $\ee_{i,\fv} $ is nonzero only if
 $i =\|\fv\|$; 
for this reason we  denote $\ee_{\|\fv\|,\fv}$  simply by $\ee_{\fv}$
for the rest of the section. 
The {\bf support} of a vector $\fv =(v_1, \ldots, v_n)$
is  $\Supp(\fv)=\{ i \, : \, v_i\ne 0\}$.
The set $\Supp(\fv)$ is said to be 
an {\bf interval} if it is of the form $\{a, a+1, \ldots, a+b\}$ for some $a$ and some $b\gs 0$,
while it is said to be a  
{\bf cyclic interval} if it is an interval or  a subset of the form 
$\{1, 2, \ldots, a, b, b+1, \ldots, n\}$ for some $a<b$.
These definitions are motivated by Lemma \ref{LemmaConsecutiveDeviations},
which plays an important role  in the rest of the paper as it narrows down the possible multidegrees of nonzero deviations.

\begin{example}
The support of $(1,2,1,0,0)$ is an interval.

The support of $(1,0,0,1,2)$ is a  cyclic interval but not an interval.

The support of $(1,0,2,0,1)$ is not a cyclic interval.
\end{example}

\begin{lemma}\label{LemmaConsecutiveDeviations}
Let $S=k[T_1,\ldots,T_n]$ with $n\gs 3$ and $\fv\in \NN^n$.
\begin{enumerate}
\item[(a)] If $\varepsilon_{\mathbf{v}}(S/\pat{n})>0$, then $\Supp(\fv)$ is an interval.
\item[(b)] If $\varepsilon_{\mathbf{v}}(S/\cyc{n})>0$, then $\Supp(\fv)$ is a cyclic interval. 
\end{enumerate}
\end{lemma}

\begin{proof}
We only prove (a) as 
the proof of (b) is the same
with straightforward modifications. 
Let $R\langle Y\rangle $ be an acyclic closure of $k$ over $R$.
We regard the elements of $R\langle Y\rangle $ as polynomials in the variables $Y$ and ${t_1, \ldots, t_n}$.
If $V \subseteq R\langle Y \rangle$ is a graded vector subspace, 
we denote by $V_{i,\fv}$ the graded component of $V$ of homological degree $i$ and internal multidegree $\fv$.
We show by induction on $i$ that the support of the multidegree of  each  variable in $Y_i$ is an interval. 
For $i=1$, the statement is clear, since the multidegrees are just the basis vectors of $\NN^n$.
The case $i=2$ follows from \cite[Lemma 5]{Berglund}, because the multidegrees of the variables in $Y_2$ are the same as those of the generators of $\pat{n}$.

Now let $i> 2$,  $y\in Y_i$ and $\fv=\mdeg(y)$, so that $\|\fv\|=i$.
Assume by contradiction that $\Supp(\fv)$ is not an interval. Then we can write $\fv=\fv_1+\fv_2$ for two nonzero vectors $\fv_1$ and $\fv_2$
such that $\Supp(\fv_1)$ and $\Supp(\fv_2)$
are  disjoint and do not contain two adjacent indices. 
By construction of $R\langle Y\rangle$,
the variable $y$ is adjoined to kill a cycle $z$ in  $R\langle Y_{\ls i-1}\rangle$ whose homology class is part of a minimal generating set of $H_{i-1}(R\langle Y_{\ls i-1}\rangle)$.
We will derive a contradiction by showing that $z$ is a boundary in 
$R\langle Y_{\ls i-1}\rangle$.
Note that $z \in R\langle Y_{\ls i-1}\rangle_{\|\fv\|-1,\fv}$,
and by induction  the variables in $Y_{\ls i-1}$ have multidegrees whose supports are intervals, 
thus we can write  
$$
z=\sum_j \left( A_j p_j + B_j q_j \right) ,
$$ 
where   $A_j \in R\langle Y_{\ls i-1} \rangle_{\|\fv_1\|,\fv_1}$ and
$B_j \in R\langle Y_{\ls i-1} \rangle_{\|\fv_2\|,\fv_2}$
 are distinct monomials
and
$p_j \in R\langle Y_{\ls i-1} \rangle_{\|\fv_2\|-1,\fv_2}$ and
$q_j \in R\langle Y_{\ls i-1} \rangle_{\|\fv_1\|-1,\fv_1}$
 are homogeneous polynomials.
Since $z$ is a cycle,  the Leibniz rule yields
\begin{align}
\begin{split}
0&=\partial(z)\\
&=\sum \partial(A_j)p_j + (-1)^{\|\fv_1\|} \sum  A_j\partial(p_j)+ \sum \partial(B_j)q_j + (-1)^{\|\fv_2\|} \sum  B_j\partial(q_j).
\end{split}
\end{align}
In the sum above, 
each monomial $A_j$ only appears in $A_j\partial(p_j)$, 
therefore $\partial(p_j)=0$.
However, 
by construction of $R\langle Y \rangle$
the homology of the DG algebra 
$R\langle Y_{\ls i-1}\rangle$
vanishes
in the homological degree 
$\|\fv_2\|-1<\|\fv\|-1=i-1$
hence $p_j$ is a boundary in $R\langle Y_{\ls i-1}\rangle$. 
Likewise, 
$q_j$ is a boundary. 

Let $P_j, Q_j$ be homogeneous polynomials such that
$\partial(P_j)=p_j$, $\partial(Q_j)=q_j$, so that   
$$z=\sum A_j\partial(P_j)+ \sum B_j\partial(Q_j).$$
Since $\partial(B_jQ_j)=\partial(B_j)Q_j+(-1)^{\|\fv_2\|}B_j\partial(Q_j),$ we have that $z$ is a boundary if and only if the cycle 
$
\sum A_j\partial(P_j)- (-1)^{\|\fv_2\|} \sum \partial(B_j)Q_j
$
is a boundary. 
In other words, 
by grouping together the terms in the two sums,
we may assume without loss of generality that 
the original cycle has the form
\begin{equation}\label{EquationFormCycle}
z= \sum p_jq_j 
\qquad
\mbox{ with }
\quad
p_j \in R\langle Y_{\ls i-1}\rangle_{\|\fv_1\|, \fv_1}
\quad
\mbox{ and }
\quad
q_j \in \partial \left( R\langle Y_{\ls i-1}\rangle_{\|\fv_2\|, \fv_2}\right)
\end{equation}
with the $q_j$  linearly independent over $k$.

Let $\{e_h\}$ be a $k$-basis of $\partial(R\langle Y_{\ls i-1}\rangle_{\|\fv_1\|, \fv_1})$ and write $\partial(p_j) = \sum \lambda_{j,h} e_h$ with $\lambda_{j,h}\in k$. Then
\begin{equation}\label{EquationLinearDependence}
0=\partial(z)=\sum \partial(p_jq_j)=\sum \partial(p_j)q_j= \sum \lambda_{j,h} e_h q_j.
\end{equation}

The boundaries form a homogeneous two-sided ideal in the subring of cycles, and thus by Equation \ref{EquationFormCycle}
in order to show that $z$ is a boundary it suffices to prove that  the $p_j$ are cycles. 
This follows from Equation \ref{EquationLinearDependence}
once we know that the set $\{e_hq_j\}$ is linearly independent.
To see this, observe that we have an embedding of graded $k$-vector spaces
$$
R\langle Y_{\ls i-1}\rangle_{\|\fv_1\|-1, \fv_1}
\otimes_k
R\langle Y_{\ls i-1}\rangle_{\|\fv_2\|-1, \fv_2}
\hookrightarrow
R\langle Y_{\ls i-1}\rangle_{\|\fv_1+\fv_2\|-2,\fv_1+\fv_2}.
$$
as the tensor product of the two monomial $k$-bases in the LHS is mapped injectively into  the monomial $k$-basis of the RHS, because 
as no $t_lt_{l+1}$ arises in the products 
by the assumption on $\fv_1$ and $\fv_2$.
This completes the proof.
\end{proof}

Given ${\bf v}=(v_1,\ldots,v_n)\in \mathbb{N}^n$, we denote the vector $(v_1,\ldots,v_n, 0)\in \mathbb{N}^{n+1}$ by $\fv^a$. We denote by $S[T_{n+1}]=k[T_1,\ldots,T_{n+1}]$, the polynomial ring in $n+1$ variables over $k$.

\begin{lemma}\label{LemmaConstantMultigradedDeviations}
Let $S=k[T_1,\ldots,T_n]$ with $n\gs 3$ and ${\bf v}=(v_1,\ldots,v_n)\in \mathbb{N}^n$, then 

\begin{enumerate}
\item[(a)] $\varepsilon_{\mathbf{v}}(S/\pat{n}) =\varepsilon_{\mathbf{v}^a}(S[T_{n+1}]/\pat{n+1})$.
\end{enumerate}
Moreover, if either $v_1=0$ or $v_n=0$ then 
\begin{enumerate}
\item[(b)] $\varepsilon_{\mathbf{v}}(S/\cyc{n}) =\varepsilon_{\mathbf{v}^a}(S[T_{n+1}]/\cyc{n+1})$,
\item[(c)] $\varepsilon_{\mathbf{v}}(S/\pat{n}) =\varepsilon_{\mathbf{v}}(S/\cyc{n})$.
\end{enumerate}
\end{lemma}

\begin{proof}
Let $R=S/\pat{n}$ and $R'=S[T_{n+1}]/\pat{n+1}$. 
From Proposition \ref{MultPoncareHilbert} and Equation \ref{MultigradedExpansion} we have 
\begin{equation}\label{EquationMultigradedPoincare}
\prod_{\|\mathbf{v}\|\, odd} (1-\fxi^\mathbf{v})^{\ve_\mathbf{v}(R)} \sum_{\fv\in\NN^n} c_\mathbf{v}(R)
\fxi^\mathbf{v}=
\prod_{\|\mathbf{v}\|\, even} (1-\fxi^\mathbf{v})^{\ve_\mathbf{v}(R)},
\end{equation}
where $c_\mathbf{v}(R)$ is the coefficient of $\fxi^\fv$ in $\HS_R(\fxi)$, 
namely
$c_\mathbf{v}(R)=0$ if $\mathbf{v}$ has two consecutive positive components and $c_\mathbf{v}(R)=1$ otherwise. 

(a) 
We proceed by induction on $\|\mathbf{v}\|$.
If $\|\mathbf{v}\|=1$ then  $\varepsilon_{\fv} (R)=\varepsilon_{\fv^a}(R')=1$, 
as these deviations corresponds to the elements $t_i$ with $1\ls i\ls n$. 
Assume now that $\|\fv\|>1$.
We reduce  Equation \ref{EquationMultigradedPoincare} modulo the ideal of $\mathbb{Z}[[\xi_1,\ldots,\xi_n]]$ generated by the monomials $\fxi^\mathbf{w}\nmid\fxi^\fv$.
Every surviving multidegree $\mathbf{u}$ other than $\mathbf{v}$ 
verifies $\|\mathbf{u}\|<\|\fv\|$,
hence by induction 
$\varepsilon_{\mathbf{u}}(R) =\varepsilon_{\mathbf{u}^a}(R')$.
Furthermore, it is clear that $c_\mathbf{u}(R)=c_{\mathbf{u}^a}(R')$. 
Hence after reducing the corresponding Equation \ref{EquationMultigradedPoincare} for $R'$ modulo the ideal of $\mathbb{Z}[[\xi_1,\ldots,\xi_{n+1}]]$ generated by the monomials 
$\fxi^\fw\nmid\fxi^{\fv^a}$
and solving the two equations for $\varepsilon_{\fv}(R)$ and $\varepsilon_{\mathbf{v}^a}(R')$ respectively, we obtain $\varepsilon_{\mathbf{v}}(R) =\varepsilon_{\mathbf{v}^a}(R')$ as desired.

(b)
The same argument as above works,
however we need to assume that either $v_1=0$ or $v_n=0$ to guarantee that $c_\mathbf{u}(S/\cyc{n})=c_{\mathbf{u}^a}(S[T_{n+1}]/\cyc{n+1})$. 

(c)
It follows  by induction and 
because
the support of a vector in  the set 
$\{\fw\,:\,\fxi^\fw\mid \fxi^\fv\}$ 
is an interval if and only if it is a cyclic interval,
provided that either $v_1$ or $v_n=0$.
\end{proof}

We say that a vector is {\bf squarefree} if its components are either 0 or 1.
In the following proposition we determine $\varepsilon_{\fv}$ for squarefree vectors $\fv$;
this result will also be useful in Section \ref{SectionKoszulHomology}.
We denote by $\f1_n$ the vector $(1,\ldots,1)\in\NN^n$.

\begin{prop}\label{PropositionSquarefreeDeviations}
Let $S=k[T_1,\ldots,T_n]$ with $n\gs 3$ and ${\bf v}\in \mathbb{N}^n$ be a squarefree vector.

\begin{enumerate}
\item[a)] If $R=S/\pat{n}$, then  $\varepsilon_{\mathbf{v}}(R)=1$ if $\Supp(\fv)$ is an interval and $\varepsilon_{\mathbf{v}}(R)=0$ otherwise.

\item[b)] If $R=S/\cyc{n}$ and $\fv\neq\f1_n$, then  $\varepsilon_{\mathbf{v}}(R)=1$ if $\Supp(\fv)$ is a cyclic interval and $\varepsilon_{\mathbf{v}}(R)=0$ otherwise. 
Furthermore, $\varepsilon_{\f1_n}(R)=n-1$.
\end{enumerate}

\end{prop}
\begin{proof}
We are going to apply  \cite[Theorem 2]{Berglund} and
we follow the notation therein.

(a)
By Lemma \ref{LemmaConsecutiveDeviations}  if $\varepsilon_{\mathbf{v}}(R)\neq 0$
then $\Supp(\fv)$ is an interval. 
Let $p=\|\fv\|$, 
the statement is clear for $p\ls 2$ by  \cite[Lemma 5]{Berglund}, 
hence we may assume $p\gs 3$. 
We have that $M_{\fv}=\{T_aT_{a+1},\ldots, T_{a+p-2}T_{a+p-1}\}$ for some $a$. Notice that any subset of $p-2$ elements of $M_{\fv}$ is either disconnected or $m_S\neq m_{M_{\fv}}$ then $\Delta'_{M_{\fv}}$ is the $p-3$ skeleton of a $p-2$ simplex. 
 Let $S^d$ be the unit sphere in $\RR^{d+1}$. Then
 \cite[Theorem 2]{Berglund} yields 
 $$
 \varepsilon_{\fv}(R)=\dim_k \widetilde\Ho_{p-3}(\Delta'_{M_{\fv}}; k)=\dim_k \widetilde\Ho_{p-3}(S^{p-3}; k)=1.
 $$  

(b)
The same argument as above works for $\fv\neq \f1_n$.
If $\fv=\f1_n$, then $M_{\f1_n}=\{T_1T_2,\ldots, T_{n}T_{1}\}$, 
and $\Delta'_{M_{\f1_n}}$ is the $n-3$ skeleton of an $n-1$ simplex, and by \cite[Theorem 2]{Berglund} we have 
$$
\varepsilon_{ \f1_n}(R)=\dim_k \widetilde\Ho_{n-3}(\Delta'_{M_{\f1_n}}; k)=
\dim_k \widetilde\Ho_{n-3}\left(\bigvee^{n-1}S^{n-3}; k\right)=n-1.
$$
\end{proof}

The next theorem determines the first $n$ deviations of the $n$-cycle and the first $n+1$ deviations of the $n$-path.

\begin{thm}\label{TheoremExistenceSequences}
Let $S=k[T_1,\ldots,T_n]$. There exist two sequences of natural numbers $\{\gamma_s\}_{s\gs 1}$ and $\{\alpha_s\}_{s\gs 1}$ such that for every $n\gs 3$

\begin{enumerate}
\item[(a)]  $\varepsilon_s(S/\pat{n}) = \gamma_sn-\alpha_s$ for  $s\ls n+1$;

\item[(b)] $\varepsilon_s(S/\cyc{n}) = \gamma_sn$ for  $s< n$ and $\varepsilon_n(S/\cyc{n}) = \gamma_nn-1$.
\end{enumerate}
\end{thm}

\begin{proof}
(b)
Fix $s\gs 1$ and for every $n\gs s$ define  the set
$$\mathcal{E}(s, n)=\big\{\mathbf{v}\in \NN^n \, : \,  \|\mathbf{v}\|=s,\, \varepsilon_{\mathbf{v}}(S/\cyc{n})>0 \big\}.
$$
Let $R=S/\cyc{n}$ and $R'=S[T_{n+1}]/\cyc{n+1}$. 
Assume first that $s<n$.
The group $\ZZ_n:=\ZZ/n\ZZ$ acts on  $\mathcal{E}{(s,n)}$ by permuting the components of a vector cyclically and
by the symmetry of $\cyc{n}$  
the
deviations are constant in every orbit. 
By Lemma \ref{LemmaConsecutiveDeviations}
the support of each $\fv \in \mathcal{E}{(s,n)}$ 
is a cyclic interval and since some component of $\fv$ is 0 we conclude that
each orbit contains exactly $n$ elements.
Similarly, every orbit in the action of $\ZZ_{n+1}$ on $\mathcal{E}{(s,\,n+1)}$ has $n+1$ elements. 
We denote orbits by $[\cdot]$ and for a given $\mathbf{v}\in\mathcal{E}{(s,\,n)}$ we denote by $\bar{\fv}=(\overline{v_1},\ldots, \overline{v_n})$ the only vector in $[\fv]$ such that $\overline{v_1}\neq 0$ and $\overline{v_n}=0$. 
The map 
$\phi\colon\,\, \mathcal{E}{(s,\,n)}/\ZZ_n  \to \mathcal{E}{(s,\,n+1)}/\ZZ_{n+1}$
defined via 
$[\bar{\mathbf{v}}] \xmapsto{\phantom{\ZZ_n}} [\bar{\mathbf{v}}^a]$
is well-defined and bijective  by Lemmas \ref{LemmaConsecutiveDeviations} and \ref{LemmaConstantMultigradedDeviations},
and moreover  $\ee_{[\fv]}(R)=\ee_{\phi([\fv])}(R')$.
Since the multigraded deviations refine the deviations we obtain

\begin{equation}\label{EquationDeviationsOrbits}
\frac{\varepsilon_s(R)}{n}=
\sum_{[\fv]\in\mathcal{E}(s,\,n)/\ZZ_n}\ve_{[\fv]}(R) \\
=\sum_{[\fv]\in\mathcal{E}(s,\,n)/\ZZ_n}\ve_{\phi([\fv])}(R')=
\frac{\varepsilon_s(R')}{n+1}.
\end{equation}
It follows that $\varepsilon_s(S/\cyc{n}) = \gamma_s n $ for every $n> s$,  for some natural number $\gamma_s$. 
Now consider the case $s=n$.
By Proposition \ref{PropositionSquarefreeDeviations}
we have $\ve_{\f1_n}(R)=n-1$ and $\ve_{\f1_n^a}(R')=1$. 
The orbit $[\f1_n]$ consists of 1 element,
while the orbit $[\f1_n^a]$ consists of $n+1$ elements, thus in this case we modify Equation \ref{EquationDeviationsOrbits} to obtain
$$
\frac{\ve_n(R)-(n-1)}{n}=
\frac{\ve_n(R') - (n+1)}{n+1}.
$$
Hence $\frac{\ve_n(R)-(n-1)}{n}= \gamma_{n}-1$ and the conclusion follows.

(a) 
Fix $s\gs 1$, and similarly define the set
$$
\mathcal{E}(s, n)=\big\{\mathbf{v}\in \NN^n \, : \,  \|\mathbf{v}\|=s,\, \varepsilon_{\mathbf{v}}(S/\pat{n})>0 \big\}.
$$ 
Let $R=S/\pat{n}$, $R'=S[T_{n+1}]/\pat{n+1}$, and assume $s\ls n$. By Lemma \ref{LemmaConsecutiveDeviations},
if $\fu=(u_1,\ldots,u_{n+1})\in\mathcal{E}(s,\,n+1)$, then either $u_1=0$ or $u_{n+1}=0$. 
The map 
$\psi\colon\,\, \mathcal{E}{(s,\,n)}  \to \mathcal{E}{(s,\,n+1)}$
defined via 
$\mathbf{v} \xmapsto{\phantom{\ZZ_n}} \mathbf{v}^a $
is injective.
By Lemma \ref{LemmaConstantMultigradedDeviations} $\Im(\psi)=\big\{\mathbf{u}=(u_1,\ldots,u_{n+1})\in\mathcal{E}(s,\,n+1)\,:\, u_{n+1}=0\big\}$ and if $\mathbf{u}\in\Im(\psi)$ then $\varepsilon_{\psi^{-1}(\fu)}(R)=\varepsilon_{\fu}(R')$.
We conclude that
$$
\varepsilon_s(R')-\varepsilon_s(R)= \sum_{\fu\in\mathcal{E}(s,\,n+1)\setminus \Im(\psi)}\ve_{\fu}(R').
$$ 
Since 
$$\mathcal{E}(s,\,n+1)\setminus \Im(\psi)=\big\{\mathbf{u}=(u_1,\ldots,u_{n+1})\in\mathcal{E}(s,\,n+1)\,:\, u_1=0,\,u_{n+1}\neq 0\big\},$$
then it follows that $\varepsilon_s(R')-\varepsilon_s(R)=\gamma_s$, by the proof of part (b) and Lemma \ref{LemmaConstantMultigradedDeviations} (c). 

Finally, let $s=n+1$. In this case, the difference 
$\varepsilon_{n+1}(R')-\varepsilon_{n+1}(R)$ is equal to the sum of $\varepsilon_{\f1_{n+1}}(R')$ and all   $\varepsilon_{\mathbf{u}}(R')$ where $\mathbf{u}=(u_1,\ldots,u_{n+1}) \in \mathcal{E}(n+1,\,n+1)$ with $u_1=0, u_{n+1}\neq 0$. By Proposition \ref{PropositionSquarefreeDeviations}, $\varepsilon_{\f1_{n+1}}(R')=1$, and thus  $\varepsilon_{n+1}(R')-\varepsilon_{n+1}(R)=\gamma_{n+1}$.

We have proved the existence of the sequence of integers $\{\alpha_s\}_{s\gs 1}$; they are  non-negative as by Lemma \ref{LemmaConstantMultigradedDeviations} (c) we have $\ee_s(S/\cyc{n})\gs\ee_s(S/\pat{n})$ for every $n> s$. 
\end{proof}

\begin{remark}\label{FurtherDeviations}
Explicit formulas for the graded Betti numbers of $\pat{n}$ and $\cyc{n}$ were found in \cite{Jacques} using Hochster's formula;
combining these formulas and Equation \ref{EquationHilbertPoincare}
one can deduce a recursion  for  the deviations. 
Through this  recursion  we noticed that some of the higher deviations also seem to be determined by the sequences $\{\alpha_s\}_{s\gs 1}$ and $\{\gamma_s\}_{s\gs 1}$. 
We observed the following patterns for cycles
\begin{eqnarray*}
\ee_{n+1}(S/\cyc{n})&=&\gamma_{n+1}n-n,\\
\ee_{n+2}(S/\cyc{n})&=&\gamma_{n+2}n-\binom{n+2}{2}+1,\\
\ee_{n+3}(S/\cyc{n})&=&\gamma_{n+3}n-\binom{n+3}{3}-\binom{n+1}{2}+1.\\
\mbox{ and the following patterns for paths:}\\
\ee_{n+2}(S/\pat{n})&=&\gamma_{n+2}n-\alpha_{n+2}+1,\\
\ee_{n+3}(S/\pat{n})&=&\gamma_{n+3}n-\alpha_{n+3}+n+2,\\
\ee_{n+4}(S/\pat{n})&=&\gamma_{n+4}n-\alpha_{n+4}+\binom{n+4}{2}-1,\\
\ee_{n+5}(S/\pat{n})&=&\gamma_{n+5}n-\alpha_{n+5}+\binom{n+5}{3}-\binom{n+3}{2}-1.
\end{eqnarray*} 
Verifying these formulas with the method in the proof of Lemma \ref{LemmaConstantMultigradedDeviations}, 
would require the explicit computation of multigraded deviations for vectors that are not squarefree. 
One possible approach is to use Equation \ref{EquationMultigradedPoincare} and proceed by induction for each multidegree; 
this is an elementary but rather intricate argument. 
Using this method we were able to verify the above identities for $\ee_{n+2}(S/\pat{n})$ and $\ee_{n+1}(S/\cyc{n})$.
\end{remark}

\begin{table}
\caption{Some Values of $\gamma_s$ and $\alpha_s$}
\centering
\begin{tabular}{l l l l l}
\hline
\hline
$s$ & $\gamma_s$ & $\alpha_s$ & $\approx \gamma_s/\gamma_{s-1}$ & $\approx \alpha_s/\alpha_{s-1}$\\
[0.5ex]
\hline
 
1 & 1 & 0 & &\\
2 & 1 & 1 & 1&\\
3 & 1 & 2 & 1&2\\
4 & 2 & 5 & 2&2.5\\
5 & 5 & 14 & 2.5&2.8\\
6 & 12 & 38 & 2.4&2.71\\
7 & 28 & 100 & 2.33 &2.62\\
8 & 68 & 269 & 2.43 &2.69\\
9 & 174 & 744 & 2.56&2.77\\
10 & 450 & 2064 & 2.59&2.77\\
11 & 1166 & 5720 & 2.59&2.77\\
12 & 3068 & 15974 & 2.63& 2.79\\
13 & 8190 & 44940 & 2.67&2.81\\
14 & 22022 & 126854 & 2.69&2.82\\
15 & 59585 & 359118 & 2.71&2.83\\
16 & 162360 & 1020285 & 2.72&2.84\\
17 & 445145 &  2907950 & 2.74&2.85\\
18 & 1226550 & 8309106 & 2.76&2.86\\
19 & 3394654 & 23796520 & 2.77&2.86\\
20 & 9434260 & 68299612 & 2.78&2.87\\
21 & 26317865 & 196420246 & 2.79&2.88\\
22 & 73662754 & 565884418 & 2.8&2.88\\
23 & 206809307 & 1632972230 & 2.81&2.89\\
24 & 582255448 & 4719426574 & 2.82&2.89\\
25 & 1643536725 & 13658698734 & 2.82&2.89\\
\hline
\end{tabular}
\label{table}
\end{table}

Using the recursive formula for deviations mentioned in Remark \ref{FurtherDeviations}, 
we compute with Macaulay2 \cite{Macaulay2} some values of the sequences $\{\alpha_s\}_{s\gs 1}$ and $\{\gamma_s\}_{s\gs 1}$, cf. Table \ref{table}.
From these values we can observe that the sequences $\{\alpha_s\}_{s\gs 1}$ and $\{\gamma_s\}_{s\gs 1}$ seemingly grow exponentially at a ratio that approaches 3. 
This observation is consistent with Theorem \ref{TheoremExistenceSequences} and the asymptotic growth of $\ee_i(S/\pat{n})$ and $\ee_i(S/\cyc{n})$ described in  \cite[4.7]{BDGMS}.

\section{Koszul homology of Koszul algebras}\label{SectionKoszulHomology}

Let $S[X]$ be a minimal model of $R$ over $S$ and let $K^S$ denote the Koszul complex of 
$S$ with respect to $\fn$.
Denoting by $\epsilon^{S[X]}: S[X]\rightarrow R$ and $\epsilon^{K^S}: K^S\rightarrow k$  the augmentation maps,
the following homogeneous DG algebra morphisms

$$\xymatrix@C=20mm{k[X]\cong S[X]\otimes_S k & S[X]  \otimes_S K^S \ar[l]_-{\id_{S[X]}\otimes_S \epsilon^{K^S}} \ar[r]^-{\epsilon^{S[X]}\otimes_S \id_{K^S}} & R\otimes_S K^S \cong K^R }$$
\noindent
are quasi-isomorphisms, i.e., they
 induce $k$-algebra isomorphisms on homology  
\[
\Tor^S(R,k)=H( k[X])\cong H(S[X]  \otimes_S K^S)\cong H( K^R)=H^R
\]
see \cite[2.3.2]{Avramov6Lectures}. Thus we have $\dim_k H^R_{i,j}= \dim_k\Tor_i^S(R,k)_{j}=\beta_{i,j}^S(R)$ for every $i,j\gs 0$, where $\beta_{i,j}^S(R)$ denote the graded Betti numbers of $R$ over $S$.
If $I$ is a monomial ideal, then $H^R$ inherits the $\NN^n$ grading from $K^R$ and for every  $\fv\in \NN^n$ we have  $\dim_k H^R_{i,\fv}=\Tor_i^S(R,k)_{\fv}=\beta_{i,\fv}^S(R)$.
In other words, the graded vector space structure of $H^R$ is completely determined by the Betti table of $R$ over $S$.

It was proved in \cite[5.1]{AvramovConcaIyengar2} that
for a Koszul algebra $R$, if $H^R_{i,j}\neq 0$ then $j\ls 2i$ and $H^R_{i,2i}=\big(H^R_{1,2}\big)^i$. 
We extend this result to the next diagonal in the Betti table.

\begin{thm}\label{KoszulSecondDiag}
Let $R$ be a Koszul algebra. Then
 for every $i\gs 2$ we have
$$
H^R_{i,2i-1}=\big(H^R_{1,2}\big)^{i-2}H^R_{2,3}.
$$
\end{thm}

\begin{proof}
We show first that 
$
k[X]_{i,2i-1}\subseteq \big(k[X]_{1,2}\big)^{i-2}k[X]_{2,3}$.
By the Koszul property and \cite[7.2.6]{Avramov6Lectures} we have $\deg(x)=|x|+1$ for every variable $x\in X$.
For every  monomial $x_{1}\cdots x_{p}\in k[X]_{i,2i-1}$, 
 we have
$$
2i-1=\deg(x_{1}\cdots x_{p}) = \deg(x_{1})+\cdots + \deg(x_{p})
$$
and
$$
i=|x_{1}\cdots x_{p}| = |x_{1}|+\cdots + |x_{p}|=\deg(x_{1})+\cdots + \deg(x_{p})-p
$$
hence $p=i-1$. 
Assume without loss of generality that $|x_j|\ls|x_{j+1}|$ for every $1\ls j\ls i-2$, then $|x_{1}|=\cdots=|x_{i-2}|=1$ and $|x_{i-1}|=2$, 
so the claim follows.

Since the model $S[X]$ is minimal,
we must have
$\partial(X_{1,2})\subseteq \mathfrak{m}S[X]$ and also $\partial(X_{2,3})\subseteq \mathfrak{m}S[X]$, 
because there cannot be a quadratic part in the differential for these low degrees (cf. \cite[7.2.2]{Avramov6Lectures}).
Hence $X_{1,2} \cup X_{2,3}\subseteq Z(k[X])$, the subalgebra of cycles of $k[X]$, therefore we have inclusions
\begin{eqnarray*}
\big(k[X]_{1,2}\big)^{i-2}k[X]_{2,3}&=& \big(Z(k[X])_{1,2}\big)^{i-2}Z(k[X])_{2,3}\subseteq Z(k[X])_{i,2i-1} \\
& \subseteq & k[X]_{i,2i-1} \subseteq \big(k[X]_{1,2}\big)^{i-2}k[X]_{2,3}.
\end{eqnarray*}
We conclude $\big(Z(k[X])_{1,2}\big)^{i-2}Z(k[X])_{2,3}= Z(k[X])_{i,2i-1}$ and the desired statement follows after going modulo $B(k[X])$, 
the ideal of boundaries of $Z(k[X])$.
\end{proof}

\begin{remark}\label{RemarkKoszulHomologyLinearStrand}
By Theorem \ref{KoszulSecondDiag} and \cite[5.1]{AvramovConcaIyengar2}
if $R$ is a Koszul algebra then the  components $H^R_{i,j}$ of the Koszul homology such that $j \gs 2i-1$ are generated by the components of bidegrees $(1,2)$ and $(2,3)$. 
It is natural then to ask whether for Koszul algebras the minimal $k$-algebra generators of
the Koszul homology have bidegrees $(i,i+1)$, corresponding to the linear strand of the Betti table of $R$;
observe that these components are necessarily minimal generators.   This question was raised by Avramov and the answer turns out to be negative:  the first example was discovered computationally by Eisenbud and Caviglia using Macaulay2.  By manipulating this example, Conca and Iyengar were led to consider edge ideals of $n$-cycles.  A family of rings for which this fails is $R=S/I(\mathcal{C}_{3k+1})$ with $k\gs 2$, whose Koszul homology has a minimal algebra generator in bidegree $(2k+1, 3k+1)$  (cf. Theorem \ref{KoszulHomologyPolygons}). 
\end{remark}

Now we turn our attention to the Koszul algebras studied in Section \ref{SectionEdgeIdeals}.
We begin with a well-known fact about resolutions of monomial ideals.

\begin{prop}
Let $I$ be a monomial ideal, $\mathbb{F}$  a multigraded free resolution of $R=S/I$, and $\mathbf{a}\in \mathbb{N}^n$.
Denote by $\mathbb{F}_{\ls \mathbf{a}}$
the subcomplex of $\mathbb{F}$ generated by the standard basis elements of multidegrees $\fv \ls \mathbf{a}$.
Then $\mathbb{F}_{\ls \mathbf{a}}$ is a free resolution of  $S/I_{\ls \mathbf{a}}$, 
where $I_{\ls \mathbf{a}}$ is the ideal of $S$ generated by the elements of $I$ with multidegrees $\fv\ls \mathbf{a}$. 
In particular,  if $I$ is  squarefree
then $\beta_{i,\fv}^S(R)\ne 0$ only for squarefree multidegrees $\fv$.
\end{prop}

Next we introduce some notation for the decomposition of squarefree vectors into  intervals and cyclic intervals (cf. Section \ref{SectionEdgeIdeals}).

\begin{definition}\label{blocks}
Let $\fv\in \mathbb{N}^n$ be a squarefree vector.
There exists a  unique minimal (with respect to cardinality) set of vectors 
$\{\fv_1,\ldots,\fv_{\tau(\fv)}\}\subset \NN^n$
 such that 
   $\Supp(\fv_j)$ is an interval for each $j$ and
$\fv=\sum_{j=1}^{\tau(\fv)}\fv_j$.   
By minimality  $\Supp(\fv_i+\fv_j)$ is not an interval if $i\neq j$. 
Define further $\iota(\fv)=\sum_{j=1}^{\tau(\fv)}\left \lfloor\frac{2\|\fv_j\|}{3}\right\rfloor$.
For example, let $\fv = (1,1,0,0,0,1,1,0,1)$ then the set is 
$$
\left\{(1,1,0,0,0,0,0,0,0),(0,0,0,0,0,1,1,0,0),(0,0,0,0,0,0,0,0,1)\right\}
$$
and we have  $\tau(\fv)=3$, $\iota(\fv)=1+1+0=2$.

Likewise, 
given a squarefree vector  $\fw\in \mathbb{N}^n$
there exists a  unique minimal set of vectors 
$\{\fw_1,\ldots,\fw_{\tilde{\tau}(\fw)}\}$
 such that 
   $\Supp(\fw_j)$ is a cyclic interval for each $j$ and
$\fw=\sum_{j=0}^{\tilde{\tau}(\fw)}\fw_j$;  
it follows that $\Supp(\fw_i+\fw_j)$ is not a cyclic interval if $i\neq j$. 
Set $\tilde{\iota}(\fw)=\sum_{j=1}^{\tilde{\tau}(\fw)} \left\lfloor\frac{2\|\fw_j\|}{3}\right\rfloor$.
For example,
let $\fw = (1,1,0,0,0,1,1,0,1)$ then the set is 
$$
\left\{(1,1,0,0,0,0,0,0,1), (0,0,0,0,0,1,1,0,0)\right\}
$$
and we have  $\tilde{\tau}(\fw)=2$, $\tilde{\iota}(\fw)=2+1=3$.
\end{definition}

Using Jacques' results \cite{Jacques}, these definitions allow us to describe completely the multigraded Betti numbers of $S/\pat{n}$ and $S/\cyc{n}$. 

\begin{prop}\label{bettionlyones}
Let $\fv$ and $\fw$ be squarefree vectors. 
Following Definition \ref{blocks}, we have:
\begin{enumerate}
\item[(a)] 
$
\beta_{i,\fv}^S(S/\pat{n})
=
1 $ if $
 \|\fv_j\|\not\equiv 1\pmod{3}$ for  $1\ls j\ls \tau(\fv)$ and $i=\iota(\fv)$;\\
 \noindent
$\beta_{i,\fv}^S(S/\pat{n})
=0$
otherwise.

\item[(b)] Assume $\fw\ne \f1_n$, then
\\ 
\noindent
$
\beta_{i,\fw}^S(S/\cyc{n})
=
1$ if $\|\fw_j\|\not\equiv 1\pmod{3}$ for $1\ls j\ls \tilde{\tau}(\fw)$ and $ i=\tilde{\iota}(\fw)$;\\
$\beta_{i,\fw}^S(S/\cyc{n})
=0$
otherwise.
 
\item[(c)] 

$
\beta^S_{i,\f1_n}(S/\cyc{n})
=1 $ if $n \equiv 1 \pmod{3}$ and $ i = \lceil \frac{2n}{3} \rceil$ or $n \equiv 2 \pmod{3}$ and $i = \tilde{\iota}({\f1_n})$;\\
$
\beta^S_{i,\f1_n}(S/\cyc{n})
= 
2 $ if $n \equiv 0 \pmod{3}$ and $ i = \tilde{\iota}({\f1_n})$;\\
$
\beta^S_{i,\f1_n}(S/\cyc{n})
= 
0$
otherwise.
\end{enumerate}

\end{prop}

\begin{proof}
Let $\FF$ be a minimal multigraded free resolution of $S/\pat{n}$. 
We prove by induction on $j$ that 
$$
\FF_{\ls\sum_{i=1}^{j}\fv_i}\cong \bigotimes_{i=1}^{j}\FF_{\ls\fv_i}.
$$ 
If $j=1$, this is trivial. 
Assume $j\gs 2$ and let $J_1=\pat{n}_{\ls \fv_j}$ and $J_2=\displaystyle\sum_{i=1}^{j-1}\pat{n}_{\ls \fv_i}$. 
Since $\Supp(\fv_j)\cap \bigcup_{i=1}^{j-1}\Supp(\fv_i)=\emptyset$ and $J_1$ and $J_2$ are monomial ideals, 
we have $0=J_1\cap J_2/J_1J_2\cong \Tor_1^S(S/J_1,S/J_2)$.
By rigidity of Tor and induction hypothesis, we conclude 
\begin{eqnarray*}
S/\sum_{i=1}^{j} \pat{n}_{\ls \fv_i}&=&S/(J_1+J_2)\cong \Tor^S(S/J_1,S/J_2)=H\Big(\FF_{\ls \fv_j}\otimes_S \bigotimes_{i=1}^{j-1}\FF_{\ls\fv_i}\Big)\\
&=&H\Big(\bigotimes_{i=1}^{j}\FF_{\ls\fv_i}\Big)
\end{eqnarray*}
which proves the claim, since minimal free resolutions are unique up to isomorphism of complexes (and the entries of the differential maps in $\FF_{\ls\sum_{i=1}^{j}\fv_i}$ and $\bigotimes_{i=1}^{j}\FF_{\ls\fv_i}$ lie in the maximal ideal of $S$ by construction). 
In particular, $\FF_{\ls\fv}\cong\bigotimes_{i=1}^{\tau(\fv)}\FF_{\ls\fv_i}$. 
Since $\FF_{\ls\fv_i}$ is a free resolution of $S/\pat{n}_{\ls \fv_i}\cong S/\pat{\|\fv_i\|}$ for every $i$, part (a) follows by \cite[7.7.34, 7.7.35]{Jacques}. 
The proof of part (b) is analogous and part (c) is a direct consequence of \cite[7.6.28, 7.7.34]{Jacques}. 
\end{proof}

Since the variables of the models $S[X]$ and $S[\widetilde{X}]$ with squarefree multidegrees play a crucial role in this section, 
we introduce here a suitable notation for them.

\begin{definition}\label{variables}
Let $S[X]$ be a minimal model of $S/\pat{n}$. 
 Proposition \ref{PropositionSquarefreeDeviations} and  \cite[Lemma 5]{Berglund}
 determine the subset of $X$ consisting of variables of squarefree multidegrees:
for every $\fv \ls \f1_n $ we have $X_{i,\fv}\ne\emptyset$ if and only if $i=\|\fv\|-1$
and $\Supp(\fv)$ is an interval,
and in this case  $\Card(X_{\|\fv\|-1,\fv})=1$. 
For this reason,
given distinct $p,q\in \{1, \ldots,n\}$  with $p<q$, 
let $x_{p,q}$ denote the only element of  $X_{\|\fv_{p,q}\|-1,\fv_{p,q}}$ where $\fv_{p,q}$ is the squarefree vector with $\Supp(\fv_{p,q})=\{p,p+1,\ldots, q\}$.
Notice $|x_{p,q}|=q-p$. 
We also denote by $x_{i,i}$ the variable $T_i$ for each $i=1, \ldots, n$.

 Let $S[\widetilde{X}]$ be a minimal model of $S/\cyc{n}$. 
Likewise, for every $\fv \ls \f1_n $, we have $\widetilde{X}_{i,\fv}\ne\emptyset$ if and only if $i=\|\fv\|-1$
and $\Supp(\fv)$ is a cyclic interval. 
Moreover, $\Card(\widetilde{X}_{\|\fv\|-1,\fv})=1$ if $\fv\ne \f1_n$ and $\Card(\widetilde{X}_{n-1,\f1_n})=n-1$. 
Given distinct $p,q\in \{1, \ldots,n\}$  with $p\not\equiv q+1\pmod{n}$, 
let $\wx_{p,q}$ denote  the only element of  $\widetilde{X}_{\|\fv_{p,q}\|-1,\fv_{p,q}}$ where $\fv_{p,q}$ denotes the squarefree vector with $\Supp(\fv_{p,q})=\{1,2,\ldots, q,p,p+1,\ldots,n\}$ if $q<p$, notice $|\wx_{p,q}|=q-p$ if $p<q$ and $|\wx_{p,q}|=n-(p-q)$ if $p>q$. 
Similarly, for each $i=1, \ldots, n$ denote by $\wx_{i,i}$ the variable $T_i$.

With an abuse of notation, 
we  denote also by $x_{p,q}$ the image of $x_{p,q} \in S[X]$ in $k[X]= k \otimes_S S[X]$,
and similarly by $\tilde{x}_{p,q}$ the image of $\tilde{x}_{p,q} \in S[\tilde{X}]$ in $k[\tilde{X}]$.
\end{definition}

\begin{example}
Let $n=7$. According to the notation just introduced we have
$$\{T_1,T_2,\ldots,T_7\}=\{x_{1,1},x_{2,2},\ldots, x_{7,7}\}=\{\wx_{1,1},\wx_{2,2},\ldots, \wx_{7,7}\},$$
$$X_{3, (0,0,1,1,1,1,0)}=\{x_{3,6}\}, \,\,\, \widetilde{X}_{3, (0,0,1,1,1,1,0)}=\{ \wx_{3,6}\}, \,\,\, \widetilde{X}_{4,(1,1,1,0,0,1,1)}=\{ \wx_{6,3} \}.$$
 \end{example}

In the next proposition we find formulas for the differential of variables with squarefree multidegree. Note that, once a multidegree $\boldsymbol\alpha$ is fixed, one can run a partial Tate process killing only cycles with multidegree bounded by $\boldsymbol\alpha$. 
The DG algebra $S[X^{\ls \boldsymbol\alpha}]$ obtained this way can be extended to a minimal model $S[X]$ of $R$ such that $X_{i, \boldsymbol\beta} = X^{\ls \boldsymbol\alpha}_{i, \boldsymbol\beta}$ for all $i > 0$, $\boldsymbol\beta \ls \boldsymbol\alpha$ componentwise. 
%(see \cite[Lemma 4]{Berglund}, whose proof applies for any choice of $\boldsymbol\alpha$). 
We apply the above strategy to compute variables with squarefree multidegree. 
We will sometimes denote the set of these variables by $X^{\text{sf}}$.

\begin{prop}\label{PropositionDifferentialSquarefreeVariable}
Following Definition \ref{variables} we have:
\begin{enumerate}
\item[(a)] There exists a minimal model $S[X]$  of $S/\pat{n}$ such that for every $p<q$
$$\partial(x_{p,q}) = \sum_{r\in\Supp(\fv_{p,q})\setminus\{q\}} (-1)^{|x_{p,r}|} x_{p,r}x_{r+1,q}.$$

\item[(b)] There exists a minimal model $S[\widetilde{X}]$  of $S/\cyc{n}$ such that for every $p\not\equiv q+1\pmod{n}$ 
\[
\partial(\wx_{p,q}) = \sum_{r\in\Supp(\fv_{p,q})\setminus\{q\}} (-1)^{|\wx_{p,r}|} \wx_{p,r}\wx_{r+1,q}\] where $\wx_{n+1,q}:=\wx_{1,q}$, and $\widetilde{X}_{n-1, \f1_n} = \{w_1, \ldots, w_{n-1}\}$ with
\[\partial(w_i) = \sum_{r=i}^{n+i-2}(-1)^{r-i}\wx_{i,r}\wx_{r+1,n+i-1}\]
where $\wx_{p, q}:=\wx_{p',q'}$ if $p \equiv p' \pmod n$, $q \equiv q' \pmod n$ and $1 \ls p', q' \ls n$.
\end{enumerate}
\end{prop}

\begin{example}
If $n=7$ then we have the following differentials   
\begin{eqnarray*}
\partial(x_{1,1}) & =&\partial(\wx_{1,1})= 0,\\
\partial(x_{1,2}) & =&\partial(\wx_{1,2})=  T_1 T_2,\\
\partial(x_{1,4}) & =& T_1 x_{2,4} - x_{1,2}x_{3,4} + x_{1,3}T_4,\\
\partial(\wx_{1,4})&=& T_1 \wx_{2,4} - \wx_{1,2}\wx_{3,4} + \wx_{1,3}T_4,\\
\partial(\wx_{4,1}) & = &T_4 \wx_{5,1} - \wx_{4,5}\wx_{6,1} + \wx_{4,6}\wx_{7,1}-\wx_{4,7}T_1.\\
\partial( w_1) & = & T_1 \wx_{2,7} - \wx_{1,2}\wx_{3,7} + \wx_{1,3}\wx_{4,7}-\wx_{1,4}\wx_{5,7}+\wx_{1,5}\wx_{6,7}-\wx_{1,6}T_7.
\end{eqnarray*}

\end{example}

\begin{proof}[Proof of Proposition \ref{PropositionDifferentialSquarefreeVariable}]
$\,$\\
(a)
We proceed by induction on $|x_{p,q}|=q-p$. 

If $q-p=1$, since $|\partial(x_{p,p+1})|=0$ and $\mdeg(\partial(x_{p,p+1}))=\fv_{p,p+1}$, we can assume $\partial(x_{p,p+1})=T_pT_{p+1}$. 
Now assume $q-p>1$;
 since $|\partial(x_{p,q})|=q-p-1$ and $\mdeg(\partial(x_{p,q}))=\fv_{p,q}$, we have
$$
\partial(x_{p,q}) = \sum_{r=p}^{q-1}\lambda_r x_{p,r}x_{r+1,q}
$$ 
for some $\lambda_r \in k$, and then by the Leibniz rule
\begin{equation}\label{diffs}
0=\partial^2(x_{p,q}) = \sum_{r=p}^{q-1}\lambda_r \big[
\partial(x_{p,r})x_{r+1,q}+
(-1)^{r-p}
x_{p,r}\partial(x_{r+1,q})  
\big].
\end{equation}
By induction hypothesis we can rewrite the RHS of Equation \ref{diffs} as
\begin{eqnarray*}
 \sum_{r=p}^{q-1}\lambda_r \left[ 
\Big(
\sum_{s=p}^{r-1} (-1)^{s-p} x_{p,s}x_{s+1,r}
\Big)
x_{r+1,q}
+
(-1)^{r-p}
x_{p,r}
\Big(
\sum_{t=r+1}^{q-1} (-1)^{t-r-1} x_{r+1,t}x_{t+1,q}
\Big)  
\right].
\end{eqnarray*}
Since the set of monomials in a given multidegree is linearly independent, the coefficient of each monomial must be 0.
For fixed $p<u<q$, the monomial $x_{p,p}x_{p+1,u}x_{u+1,q}$ appears in the first sum when $r=u, s=p$ and in the second sum when $r=p, t=u$, hence
$$
\lambda_u(-1)^{p-p}+ \lambda_p (-1)^{p-p}(-1)^{u-p-1} =0.
$$
Therefore, $\lambda_u =(-1)^{u-p} \lambda_p$, for every $p\ls u\ls q-1$. Since $S[X]$ is a minimal model we have $\lambda_p\ne0$, 
thus we can assume $\lambda_p=1$ and the conclusion follows.

(b)
By the proof of part (a) we know there exists a DG algebra satisfying the first part of the claim: let $S[\widetilde{X}^{\text{sf}}_{\ls n-2}]$ be this DG algebra, which is obtained by adding to $S$ all the variables $\wx_{p, q}$ where $p\not\equiv q+1\pmod{n}$.
We now show a possible choice of the variables in homological degree $n-1$ and internal multidegree $\f1_n$.
For $i \in \{1, \ldots, n-1\}$ let
\[
z_i = \sum_{r=i}^{n+i-2}(-1)^{r-i}\wx_{i,r}\wx_{r+1,n+i-1}.
\]
One checks easily that each $z_i$ is a cycle in $S[\widetilde{X}^{\text{sf}}_{\ls n-2}]$. Moreover, the $z_i$'s are linearly independent over $k$, since the monomial $\wx_{j, n-1}\wx_{n, j-1}$ appears only in $z_j$ for any $j \in \{1, \ldots, n-1\}$.
Note that there exists no nonzero boundary of $S[\widetilde{X}^{\text{sf}}_{\ls n-2}]$ having homological degree $n-2$ and internal multidegree $\f1_n$: if $b$ were such a boundary, then there would exist $w$ in $S[\widetilde{X}^{\text{sf}}_{\ls n-2}]$ having homological degree $n-1$, internal multidegree $\f1_n$ and such that $\partial(w) = b$. Since $\mdeg(\wx_{p, q})=\fv_{p, q}$ and $|\wx_{p, q}| = \Card(\text{Supp}(\fv_{p, q}))-1$, no such $w$ can be obtained as a linear combination of products of some $\wx_{p, q}$'s (the objects obtained that way and having multidegree $\f1_n$ must have homological degree at most $n-2$).
Let $cls(z_i)$ be the homology class of $z_i$. We now claim that \[H_{n-1, \f1_n}(S[\widetilde{X}^{\text{sf}}_{\ls n-2}]) = \langle cls(z_1), \ldots, cls(z_{n-1})\rangle.\] Since $\ee_{n-1, \f1_n} = n-1$ by Proposition \ref{PropositionSquarefreeDeviations}, it suffices to prove that the homology classes of the $z_i$'s are minimal generators of $\langle cls(z_1), \ldots, cls(z_{n-1})\rangle$.
Suppose $z_i - \sum_{j \neq i}\mu_jz_j$ equals a boundary $b$ for some $\mu_j \in S[\widetilde{X}^{\text{sf}}_{\ls n-2}]$: since all $z_i$'s have homological degree $n-2$ and  multidegree $\f1_n$, we can suppose the $\mu_j$'s all lie in $k$ and $b$ is homogeneous of multidegree $\f1_n$. Since such a boundary is forced to be zero and the $z_i$'s are $k$-linearly independent, we get a contradiction.
\end{proof}

%\begin{prop}\label{cyclemodel}
%There exists a minimal model $S[\widetilde{X}]$  of $S/\cyc{n}$ such that for every $p\not\equiv q+1\pmod{n}$ 
%\[
%\partial(\wx_{p,q}) = \sum_{r\in\Supp(\fv_{p,q})\setminus\{q\}} (-1)^{|\wx_{p,r}|} \wx_{p,r}\wx_{r+1,q}\] where $\wx_{n+1,q}:=\wx_{1,q}$ and $\widetilde{X}_{n-1, \f1_n} = \{w_1, \ldots, w_{n-1}\}$ with
%\[\partial(w_i) = \sum_{r=i}^{n+i-2}(-1)^{r-i}\wx_{i,r}\wx_{r+1,n+i-1}\]
%where $\wx_{p, q}:=\wx_{p',q'}$ if $p \equiv p' \pmod n$, $q \equiv q' \pmod n$ and $1 \ls p', q' \ls n$.
%\end{prop}
Next we introduce a compact way to denote  monomials of $S[X]$ and $S[\widetilde{X}]$ with squarefree multidegrees. 
As for the variables $x_{p,q}$ and $\wx_{p,q}$,
we use the same symbol to denote monomials in $S[X]$ (resp. $ S[\widetilde{X}]$) and their images in
$k[X]$ (resp. $k[\widetilde{X}]$).

\begin{definition}\label{monomials}
Let $S[X]$ and $S[\widetilde{X}]$ be as in Proposition \ref{PropositionDifferentialSquarefreeVariable}. 
Given $N\in \NN$ and a pair of sequences of natural numbers $P=\{p_i\}_{i=1}^N$ and $Q=\{q_i\}_{i=1}^N$ such that $1\ls p_1<q_1\ls n$ and $q_i<p_{i+1}<q_{i+1}\ls n$ for each $i\ls N-1$,
 we consider the monomial of $k[X]$
$$
\mathcal{B}_{P,Q}=\prod_{i=1}^{N}x_{p_i,q_i}.
$$
Similarly, given  $P=\{p_i\}_{i=1}^N$ and $Q=\{q_i\}_{i=1}^N$ such that $1\ls p_1<q_1<p_1+n-1<2n$ and $q_i<p_{i+1}<q_{i+1}< p_1+n$ for each $i\ls N-1$,
 we consider the monomial of $k[\widetilde{X}]$
$$
\widetilde{\mathcal{B}}_{P,Q}=\prod_{i=1}^{N}\wx_{p_i,q_i},
$$
where if $p_i>n$ or $q_i>n$ we set
$\wx_{p_i,q_i}:=\wx_{p'_i,q'_i}$ 
with
$p'_i\equiv p_i, q'_i \equiv q_i \pmod{n}$ and $1 \ls p'_i, q'_i \ls n$.

For each pair of sequences $(P,Q)$ as above define 
$$
\Gamma_{P,Q}=\{i>1\,:\,p_i=q_{i-1}+1\}
$$ 
and for each $i\in \Gamma_{P,Q}$ denote by $P(i)$ and $Q(i)$ the sequences of $N-1$ elements obtained by deleting $p_i$ from $P$ and $q_{i-1}$ from $Q$, respectively.
\end{definition}

\begin{example}\label{sequences}
Let $n=16$, $N=5$, $P=\{1,4, 7, 11,14\}$ and $Q=\{3,5,9,13,15\}$. Then 
$$
\mathcal{B}_{P,Q}=x_{1,3}x_{4,5}x_{7,9}x_{11,13}x_{14,15}.
$$ 
In this case, $\Gamma_{P,Q}=\{2,5\}$, $P(2)=\{1, 7, 11,14\}$, $Q(2)=\{5,9,13,15\}$, $P(5)=\{1,4, 7, 11\}$, $Q(5)=\{3,5,9,15\}$.   
\end{example}

\begin{remark}\label{coeffs}
Notice that for each $j \in \Gamma_{P,Q}$ we have
$$
 \mdeg(\mathcal{B}_{P(j),Q(j)})=\mdeg(\mathcal{B}_{P,Q})
\qquad
\text{and}
\qquad
|\mathcal{B}_{P(j),Q(j)}|=|\mathcal{B}_{P,Q}|+1,
$$
$$
 \mdeg(\widetilde{\mathcal{B}}_{P(j),Q(j)})=\mdeg(\widetilde{\mathcal{B}}_{P,Q})
\qquad
\text{and}
\qquad
|\widetilde{\mathcal{B}}_{P(j),Q(j)}|=|\widetilde{\mathcal{B}}_{P,Q}|+1.
$$
Now suppose one of the following holds:
\begin{itemize}
\item $\mdeg(\mathcal{B}_{P,Q}) < \f1_n$ (resp. $\mdeg(\widetilde{\mathcal{B}}_{P,Q}) < \f1_n$);
\item $\mdeg(\mathcal{B}_{P,Q}) = \f1_n$ and $N > 1$;
\item $\mdeg(\widetilde{\mathcal{B}}_{P,Q}) = \f1_n$ and $N > 2$.
\end{itemize}
Then, if the coefficient of $\mathcal{B}_{P,Q}$ (resp. $\widetilde{\mathcal{B}}_{P,Q}$) in the differential of another monomial of $k[X]$ is nonzero, 
one has that this monomial must be $\mathcal{B}_{P(i),Q(i)}$ (resp. $\widetilde{\mathcal{B}}_{P(i),Q(i)}$) for some $i\in \Gamma_{P,Q}$.
Consider $\partial(\sum_{i\in\Gamma_{P,Q}} \lambda_{P(i),Q(i)}\mathcal{B}_{P(i),Q(i)})$ for some $\lambda_{P(i),Q(i)}\in k$, 
then by Proposition \ref{PropositionDifferentialSquarefreeVariable} the coefficient of $\mathcal{B}_{P,Q}$ in this expression is
$$
\sum_{i\in\Gamma_{P,Q}}(-1)^{\sum_{j=1}^{i-1}(q_j-p_j)}\lambda_{P(i),Q(i)}.
$$ 
Similarly,  
the coefficient of $\widetilde{\mathcal{B}}_{P,Q}$ in $\partial(\sum_{i\in\Gamma_{P,Q}} \lambda_{P(i),Q(i)}\widetilde{B}_{P(i),Q(i)})$ is $$\sum_{i\in\Gamma_{P,Q}}(-1)^{\sum_{j=1}^{i-1}(q_j-p_j)}\lambda_{P(i),Q(i)}.$$ 
\end{remark}

In the following lemma we show that the homology classes of some of the  monomials 
introduced in Definition \ref{monomials} are nonzero. 

\begin{lemma}\label{notboundary}
Let $n > 3$ and let $P=\{p_i\}_{i=1}^N$ and $Q=\{q_i\}_{i=1}^N$ be two sequences of natural numbers as in Definition \ref{monomials}. Assume the following conditions $(\star)$ are satisfied:
\begin{itemize}
\item  $q_i-p_i\in \{1,2\}$ for every $i$, 
\item if $q_i-p_i=1$ then either $i=N$ or $q_i<p_{i+1}-1.$ 
\end{itemize} 
Then $\mathcal{B}_{P,Q}$ (resp. $\widetilde{\mathcal{B}}_{P,Q}$) is a cycle but not a boundary in $k[X]$ (resp. $k[\widetilde{X}]$).
\end{lemma}
\begin{proof}
Assume first we are in the hypotheses of Remark \ref{coeffs}.
We give the proof for  $k[X]$, and the one for $k[\widetilde{X}]$ is analogous. 
From Proposition \ref{PropositionDifferentialSquarefreeVariable}
we get that for every $p$, the variables $x_{p,p+2}$ and $x_{p,p+1}$ are cycles, 
so ${\mathcal{B}}_{P,Q}$ is a cycle as well. 
Now we show that it is not equal to the differential of any linear combinations of monomials of $k[X]$. Consider
\begin{equation}\label{coefficient}
\partial\left(\sum\lambda_{P',Q'}{\mathcal{B}}_{P',Q'}\right)
\end{equation} 
for some $\lambda_{P',Q'}\in k$ with the sum ranging over all the monomials in the same multidegree of ${\mathcal{B}}_{P,Q}$ 
and homological degree one higher, i.e., one variable less. For every subset $\psi\subseteq \Gamma_{P,Q}$, let
$P^{\psi}=\{p_i^{\psi}\}_{i=1,\ldots, N}$ and $Q^{\psi}=\{q_i^{\psi}\}_{i=1,\ldots, N}$ where $p^{\psi}_i=p_i-1$ if $i\in\psi$ and $p^{\psi}_i=p_i$ otherwise; and $q^{\psi}_{i}=q_{i}-1$ if $i+1\in\psi$ and $q^{\psi}_{i}=q_{i}$ otherwise (see Example \ref{shiftings}). By Remark \ref{coeffs}, the coefficient of ${\mathcal{B}}_{P^{\psi},Q^{\psi}}$ 
in Equation \ref{coefficient} is 

\begin{equation}\label{lambda}
\sum_{i\in\Gamma_{P^{\psi},Q^{\psi}}}(-1)^{\sum_{j=1}^{i-1}(q_j^{\psi}-p_j^{\psi})}\lambda_{P^{\psi}(i),Q^{\psi}(i)}=\sum_{i\in\Gamma_{P,Q}}(-1)^{\sum_{j=1}^{i-1}(q_j^{\psi}-p_j^{\psi})}\lambda_{P^{\psi}(i),Q^{\psi}(i)}.
\end{equation}
For each $i\in\Gamma_{P,Q}$, the coefficient of $\lambda_{P^{\psi}(i),Q^{\psi}(i)}$ in Equation \ref{lambda} is 
\begin{eqnarray*}
&(-1)^{\sum_{j=1}^{i-1}(q_j-p_j)} &\text{ if } i\not\in \psi
\\
&(-1)^{\sum_{j=1}^{i-1}(q_j-p_j)-1} &\text{ if } i\in \psi. 
\end{eqnarray*}

We claim that the sum of the coefficients of the monomials ${\mathcal{B}}_{P^{\psi},Q^{\psi}}$ in Equation \ref{coefficient},
 considering all  possible subsets $\psi\subseteq \Gamma_{P,Q}$,
  is equal to zero. 
This holds because if $i\in\Gamma_{P,Q}\setminus \psi$ then $P^{\psi}(i)=P^{\psi\cup\{i\}}(i)$ and $Q^{\psi}(i)=Q^{\psi\cup\{i\}}(i)$ and hence each coefficient $\lambda_{P^{\psi}(i),Q^{\psi}(i)}$ 
appears twice with opposite signs (see Example \ref{shiftings}). 
In particular, Equation \ref{coefficient} will never be equal to ${\mathcal{B}}_{P,Q}={\mathcal{B}}_{P^{\emptyset},Q^{\emptyset}}$, finishing the proof.

Assume now that the hypotheses of Remark \ref{coeffs} are not satisfied. If $N = 1$ and $\mdeg(\mathcal{B}_{P, Q})$ equals $\f1_n$, then $n$ equals either $2$ or $3$, against our assumption.

If $\mdeg(\widetilde{\mathcal{B}}_{P,Q}) = \f1_n$ and $N=2$, then the conditions $(\star)$ imply that $n$ is either $5$ or $6$. Then, knowing by Proposition \ref{PropositionDifferentialSquarefreeVariable} (b) the differential of  $w_1, \ldots, w_{n-1}$, one can check the claim by hand by slightly modifying the idea of the main case.
\end{proof}

\begin{example}\label{shiftings}
For the sequences in Example \ref{sequences}, the possible sets $\psi$ are $\emptyset$, $\psi_1=\{2\}$, $\psi_2=\{5\}$, and $\psi_3=\{2,5\}$. Notice that $${\mathcal{B}}_{P^{\psi_1},Q^{\psi_1}}=x_{1,2}x_{3,5}x_{7,9}x_{11,13}x_{14,15},$$
$${\mathcal{B}}_{P^{\psi_2},Q^{\psi_2}}=x_{1,3}x_{4,5}x_{7,9}x_{11,12}x_{13,15},$$
$${\mathcal{B}}_{P^{\psi_3},Q^{\psi_3}}=x_{1,2}x_{3,5}x_{7,9}x_{11,12}x_{13,15}.$$
 Therefore, $${\mathcal{B}}_{P^{\psi_1}(2),Q^{\psi_1}(2)}=x_{1,5}x_{7,9}x_{11,13}x_{14,15}={\mathcal{B}}_{P^{\emptyset}(2),Q^{\emptyset}(2)},$$
 $${\mathcal{B}}_{P^{\psi_2}(5),Q^{\psi_2}(5)}=x_{1,3}x_{4,5}x_{7,9}x_{11,15}={\mathcal{B}}_{P^{\emptyset}(5),Q^{\emptyset}(5)},$$
$${\mathcal{B}}_{P^{\psi_1}(5),Q^{\psi_1}(5)}=x_{1,2}x_{3,5}x_{7,9}x_{11,15}={\mathcal{B}}_{P^{\psi_3}(5),Q^{\psi_3}(5)},$$
 $${\mathcal{B}}_{P^{\psi_2}(2),Q^{\psi_2}(2)}=x_{1,5}x_{7,9}x_{11,12}x_{13,15}={\mathcal{B}}_{P^{\psi_3}(2),Q^{\psi_3}(2)}.$$

\end{example}

We are now ready to present the main theorem of this section. 

\begin{thm}\label{KoszulHomologyPolygons}
Let $S=k[T_1,\ldots,T_n]$ with $n\gs 3$.
\begin{enumerate}
\item[(a)]  If $R =S/\pat{n}$, then the $k$-algebra $H^R$ is generated  by $H^R_{1,2}$ and $H^R_{2,3}$.
\item[(b)] 
If $R = S/\cyc{n}$,
 then the $k$-algebra  $H^R$ is generated by $H^R_{1,2}$ and $ H^R_{2,3}$ if and only if  $n\not \equiv 1\pmod{3}$.
If $n \equiv 1\pmod{3}$ then for any $0\ne z\in H^R_{\lceil\frac{2n}{3}\rceil, n}$ the $k$-algebra  $H^R$ is generated by $H^R_{1,2}$, $H^R_{2,3}$, and $z$.  
\end{enumerate}
\end{thm}

\begin{proof}
(b)
Let $\f1_n\ne\fw\in \NN^n$ be such that $\beta^S_{\tilde{\iota}(\fw),\fw}(S/\cyc{n})\ne 0$.
Following Definition \ref{blocks}, assume without loss of generality that the vectors $\fw_j$ are ordered increasingly according to $\min(\Supp(\fw_j))$.

By Proposition \ref{bettionlyones} (b), we get $\beta^S_{\tilde{\iota}(\fw),\fw}(S/\cyc{n})=1$, and furthermore there exists a unique pair of sequences $P$ and $Q$ satisfying the hypothesis of Lemma \ref{notboundary},
with $p_1=\min(\Supp(\fw_1))$ and $\mdeg(\widetilde{B}_{P,Q})=\fw$.
Notice $|\widetilde{B}_{P,Q}|=\tilde{\iota}(\fw)$ and by Lemma \ref{notboundary} the image of $\widetilde{B}_{P,Q}$ in $H^R$ is nonzero, 
hence it is a $k$-basis of $H^R_{\widetilde{\iota}(\fw),\fw}$. 
By construction the homology class of $\widetilde{B}_{P,Q}$  is generated by $H^R_{1,2}$ and $H^R_{2,3}$, hence it only remains to consider the case $\fw=\f1_n$.

{\bf Case 1: \boldmath$n\equiv 2\pmod{3}$\unboldmath}

If $n\equiv 2\pmod{3}$, by Proposition \ref{bettionlyones} (c) we have $\beta^S_{\tilde{\iota}(\f1_n),\f1_n}(S/\cyc{n})=1$. 
Defining $P$ and $Q$ as above, we conclude that $H^R_{\tilde{\iota}(\f1_n),\f1_n}$ is generated by  $H^R_{1,2}$ and $H^R_{2,3}$. 

{\bf Case 2: \boldmath$n\equiv 0\pmod{3}$\unboldmath}

If $n\equiv 0\pmod{3}$ then $\beta^S_{\tilde{\iota}(\f1_n),\f1_n}(S/\cyc{n})=2$. 
If $n = 3$ the claim is trivial, since in this case the only $\mathbb{N}$-graded nonzero Betti numbers of $S/\cyc{n}$ are $\beta^S_{0,0}$, $\beta^S_{1,2}$ and $\beta^S_{2,3}$. If $n > 3$, 
we define the  sequences $P=\{1,\,4,\,\ldots,\, n-2\}$,  $Q=\{3,\,6,\,\ldots,\, n\}$, $P'=\{2,\,5,\,\ldots,\, n-1\}$, and  $Q'=\{4,\,7,\,\ldots,\, n+1\}$. From Lemma \ref{notboundary} we know that $\widetilde{B}_{P,Q}$ and $\widetilde{B}_{P',Q'}$ are cycles of $k[\widetilde{X}]$. 
Suppose a linear combination $\lambda_{P,Q}\widetilde{B}_{P,Q}+\lambda_{P',Q'}\widetilde{B}_{P',Q'}$ is a boundary,
we may proceed exactly as in the proof of Lemma \ref{notboundary} to conclude $\lambda_{P,Q}=0$. 
Therefore $\lambda_{P',Q'}{\mathcal{B}}_{P',Q'}$ is a boundary, this forces $\lambda_{P',Q'}=0$ again by Lemma \ref{notboundary}. 
Hence the homology classes of  $\widetilde{B}_{P,Q}$ and $\widetilde{B}_{P',Q'}$  are linearly independent. 
This shows that $H^R_{\widetilde{i}(\f1_n),\f1_n}$ is generated by $H^R_{1,2}$ and $H^R_{2,3}$.

{\bf Case 3: \boldmath$n\equiv 1 \pmod{3}$\unboldmath}

If $n\equiv 1 \pmod{3}$, then $\beta^S_{\lceil\frac{2n}{3}\rceil,\f1_n}(S/\cyc{n})=1$. 
Suppose $H^R_{\lceil\frac{2n}{3}\rceil,\f1_n}$ is the product of elements in smaller homological degrees, 
then there exists a set $\{\fu_1,\ldots,\fu_p\}\subset\NN^n$ such that $\f1_n=\sum_{i=1}^p\fu_i$ with $\Supp(\fu_i)$ being a cyclic interval for every $i$ and $\lceil\frac{2n}{3}\rceil=\sum_{i=1}^{p}\lfloor\frac{2\|\fu_i\|}{3}\rfloor$. 
This contradicts the fact that $\sum_{i=1}^{p} \|\fu_i\|=n\not\equiv 0\pmod{3}$. Hence, $H^R_{\lceil\frac{2n}{3}\rceil,\f1_n}$ contains minimal algebra generators of $H^R$. 
The conclusion follows.

The above proof works also for (a), if $n > 3$. 
Case 1 follows likewise via Proposition \ref{bettionlyones} (a).
Case 2 is simpler since $\beta^S_{i(\f1_n),\f1_n}(S/\pat{n})=1$ for $n\equiv 0 \pmod{3}$ and Case 3 is trivial since $\beta^S_{i(\f1_n),\f1_n}(S/\pat{n})=0$ for $n\equiv 1 \pmod{3}$. 
Finally, for $n =3$ the claim is trivial, since in these case the only $\mathbb{N}$-graded nonzero Betti numbers of $S/\pat{n}$ are $\beta^S_{0,0}$, $\beta^S_{1,2}$ and $\beta^S_{2,3}$.
\end{proof}

By Theorem \ref{KoszulHomologyPolygons}  
we can determine, more generally, the $k$-algebra generators of $H^R$ when $R=S/I(\mathcal{G})$ and $\mathcal{G}$ is a graph whose vertices have degree at most 2. 
Such graphs are disjoint unions of paths and cycles, hence it follows that $R$ is of the form
$$
R \cong S_1/\cyc{n_1} \otimes_k \cdots \otimes_k  S_a/\cyc{n_a} \otimes_k S_{a+1}/\pat{n_{a+1}} \otimes_k \cdots \otimes_k  S_b/\pat{n_b}
$$
where each $S_i$ is a polynomial ring in $n_i$ variables over $k$, yielding
 an isomorphism of $k$-algebras
$$
H^R  \cong H^{ S_1/\cyc{n_1}} \otimes_k \cdots \otimes_k  H^{S_a/\cyc{n_a}} \otimes_k H^{S_{a+1}/\pat{n_{a+1}}} \otimes_k \cdots \otimes_k  H^{S_b/\pat{n_b}}.
$$
Notice that the ideals considered here are not prime. 
In fact, we know no examples of domains $R$ for which the question in Remark \ref{RemarkKoszulHomologyLinearStrand} has a negative answer, therefore we conclude the paper with the following:

\begin{question}
Is there a Koszul algebra $R$ which is a domain and whose Koszul homology $H^R$ is not generated as a $k$-algebra in the linear strand? 
\end{question}

\section*{Acknowledgements}

This project originated during the workshop Pragmatic 2014 in Catania.
The authors would like to express their sincere gratitude to the organizers Alfio Ragusa, Francesco Russo, and
Giuseppe Zappal\`a and to the lecturers Aldo Conca, Srikanth Iyengar, and Anurag Singh.
The authors are especially grateful to the first two lecturers  for suggesting this topic and for several helpful discussions.

\end{document}